\title{Quantifying knowledge with a new calculus for belief functions - a generalization of probability theory}
\date{\today}
\author{Timber Kerkvliet and Ronald Meester \\ \\ VU University Amsterdam}

\documentclass[11pt,twoside,a4paper]{article}

\usepackage{amsthm}
\usepackage{amsmath}
\usepackage{amssymb}

\numberwithin{equation}{section}

\newtheorem{theorem}{Theorem}[section]
\newtheorem{lemma}[theorem]{Lemma}

\theoremstyle{definition}
\newtheorem{definition}[theorem]{Definition}

\theoremstyle{remark}
\newtheorem{example}[theorem]{Example}
\newtheorem{remark}[theorem]{Remark}

\newcommand{\B}{{\rm Bel}}
\newcommand{\E}{{\rm Exp}}

\begin{document}
\maketitle

\begin{abstract} 
We first show that there are practical situations in for instance forensic and gambling settings, in which applying classical probability theory, that is, based on the axioms of Kolmogorov, is problematic. We then introduce and discuss Shafer belief functions. Technically, Shafer belief functions generalize probability distributions. Philosophically, they pertain to individual or shared knowledge of facts, rather than to facts themselves, and therefore can be interpreted as generalizing epistemic probability, that is, probability theory interpreted epistemologically. Belief functions are more flexible and better suited to deal with certain types of uncertainty than classical probability distributions. We develop a new calculus for belief functions which does not use the much criticized Dempster's rule of combination, by generalizing the classical notions of conditioning and independence in a natural and uncontroversial way. Using this calculus, we explain our rejection of Dempster's rule in detail. We apply the new theory to a number of examples, including a gambling example and an example in a forensic setting. We prove a law of large numbers for belief functions and offer a betting interpretation similar to the Dutch Book Theorem for probability distributions.
\end{abstract}


\section{Introduction}

In many situations, the classical Kolmogorov axioms for probability lead to a very useful theory, with many connections to other branches of mathematics and with numerous important applications. The axioms themselves can be justified in many ways, for instance via a frequentistic interpretation of probabilities. In such a frequentistic interpretation, we take relative frequencies in repeated experiments as the motivation and justification of the axioms. Other justifications for the axioms of Kolmogorov are possible as well, see e.g.\ \cite{goos} and references below. 

This is not to say, however, that the Kolmogorov axioms should be the only and exclusive way to deal with uncertainty. Especially when uncertainty is interpreted epistemologically, that is, relating to knowledge of facts rather than to facts themselves, it is not always the case that the classical axiom of additivity adequately describes the situation at hand. For instance, in a legal or forensic setting it has already been debated for several decades as to what extent the classical theory of probability and alternatives to it, are useful and/or suitable for assessing the value of evidence, see e.g.\ \cite{fried96}, \cite{dawid}, \cite{rob_vig95}, \cite{ait95}, \cite{sjerps00}. There is a number of aspects about modeling epistemic uncertainty for which the classical approach is problematic, and we start our contribution with a short discussion of these.

First it has been observed by many that the classical theory cannot distinguish between lack of belief and disbelief. Here, disbelief is associated with evidence indicating the negation of a proposition, whereas lack of belief is associated with not having evidence at all. As Shafer \cite{shafer76} puts it, the classical theory does not allow one to withhold belief from a proposition without according that belief to the negation of the proposition. When we want to apply a theory of probabilities to legal issues, this becomes a relevant issue. Indeed, if certain exculpatory evidence in a case is dismissed, then this may result in less belief in the innocence of the suspect, but it gives no further indication for guilt. 

The second shortcoming of the classical theory is its inability to model ignorance on an individual level in situations where only group information is available. Here is a classical example. 

\begin{example} ({\em The island problem})
\label{ex2}
In the classical version of the island problem (see e.g.\ \cite{sloomee11} and \cite{baldon95}) a crime has been committed on an island, making it a certainty that an inhabitant of the island committed it. In the absence of any further information, the classical point of view is to assign a uniform prior probability over all inhabitants concerning the question who is the culprit. However, this does not correspond to our knowledge. We know for sure that someone in the population committed the crime, but have no further belief about any individual. It would, therefore, be unreasonable to assign any further individual belief to the guilt of an individual, other than the fact that the population to which he or she belongs receives belief 1. This last fact distinguishes members from the population from individuals outside it. However, the combination of assigning degree of belief 1 to the collection of all inhabitants and 0 to each individual is impossible under the classical axioms of probability, although this may be exactly the prior one wants to impose.
\end{example}

We will apply the theory we are about to develop to this example in Section \ref{subsec:forensic}. There we will see that using an uninformative prior, which is possible in our theory, leads to a different result than using a uniform prior. The fact that these priors lead to different results confirms that these priors are really distinct: a uniform prior is not a prior representing ignorance, and using a uniform prior does not lead to the same results as using a prior that does represent ignorance.

This example suggest that the usual additivity, that is, $P(A) + P(B)= P(A \cup B)$ if $A$ and $B$ are disjoint, is not always desirable when $P$ is interpreted epistemologically, as is often the case in legal or forensic settings. In such a setting, one needs to model uncertainty on the level of knowledge that one actually has, and there is no reason to suppose that individual or shared knowledge can always be adequately described by a classical underlying probability distribution, known or unknown. 

We next give a gambling example which is, like Example \ref{ex2}, just another example of the inability of probability distributions to express ignorance. 

\begin{example}
\label{ex3}
Suppose a fair coin is flipped. However, with probability $p>0$, the person flipping the coin gets the opportunity to change the outcome of the flip. In ignorance about the way the person makes his or her decisions, we can not give a probability distribution describing the outcome of this process. We can not even assume that there necessarily exists a probability distribution describing the decisions of this person. Our theory will, nevertheless, allow us to make quantitative statements on which we can base gambling strategies in this situation, see Section \ref{examplegambling}.
\end{example}

Already back in the seventies of the previous century, there has been an attempt by Shafer \cite{shafer76} to develop a theory of probabilities outside the realm of the axioms of Kolmogorov. He introduced the concept of a belief function, which is a generalization of a probability distribution. Belief functions are not necessarily additive and allow for the flexibility that our examples ask for. Based on the concept, Shafer also introduced a calculus for belief functions centered around the so called Dempster's rule of combination. However, his attempt has been criticized fiercely (references below), for various and good reasons, and nowadays belief functions are hardly used, if at all, in mainstream applied probability.

In this article, we aim to re-develop the theory of belief functions, using the basic concept of Shafer, but setting up a new calculus without using Dempster's rule. We think that our revision takes away the three reasons why people have rejected Shafer's belief functions before, which we now discuss.

The first important obstacle is reported by Shafer himself in \cite{shafer81}. Probability has a betting interpretation based on the Dutch Book Theorem, which traces back to Ramsey \cite{ramsey31} and de Finetti \cite{fin37}. Shafer writes that many of his critics rejected them because of the lack of a suitable betting interpretation for belief functions. Shafer himself argues in \cite{shafer81} that no such behavioral interpretation is necessary. We do not have to follow that line of reasoning, because in Section \ref{subsec:betting} we present a betting interpretation based on a characterization of belief functions (Theorem \ref{thm:bettingbelief}), much like the betting interpretation of probability based on the Dutch Book Theorem. 

What is probably the biggest concern about Shafer's belief functions, see e.g.\ \cite{schum94} and \cite{fs86}, is Dempster's rule of combination. This rule of combination is supposed to describe how different belief functions that are based on `independent evidence' should be combined into a new one. Let us be clear about this point: we also reject Dempster's rule and the calculus that stems from it. In Section \ref{sec:dempster} we explain how the rule confuses `independence' of evidence and `independence' of phenomena. While the troubled notion of `independent evidence' has no place in our new theory, we do have a mathematical notion of `independent phenomena' as a generalization of independence in probability theory (see Section \ref{sec:independence}). Despite our rejection of Dempster's rule, in the current article we do develop a very useful calculus of belief functions without using Dempster's rule. The only thing we need is a proper rule for conditioning, and this is much less controversial, if at all. This should take care of the points raised in \cite{schum94} and \cite{fs86}.

The third concern about belief functions concerns the question whether or not they represent knowledge or belief in a meaningful epistemological way. Pearl in \cite{pearl}, for instance, questions whether or not belief functions respect some `rules' of reasoning. For the most part, Pearl's criticism does not apply to our theory and how we want to use it, the exception being a point about belief updating. In Section \ref{sec:conditioning}, after we introduced our rule of conditioning, we address this point.

The theory of belief functions which we are about to re-develop is, foremost, an epistemic theory on the level of individual or shared knowledge. The fact that we undertake this effort implicitly implies that we think that there are many situations in which classical epistemic interpretations fall short; see our examples above. Our motivation for reviving the theory of belief functions hence lies in the fact that knowledge (or lack thereof) does not always fit into the classical framework, and that the classical axioms of probability theory are not always suitable for an epistemic interpretation. For an overview of classical epistemic interpretations we refer to \cite{gala09} and the references therein. 

On the technical level, belief functions are a generalization of classical probability theory in the sense that any probability distribution is also a belief function. So, if there are reasons to assume that the quantities we want to describe can be adequately modeled by assuming an underlying classical probability distribution, then the theory of belief functions allows for that. In other words, we lose nothing.  

For the applications of the theory to forensic examples, we refer to our companion paper \cite{kerkmee15} that has an in-depth discussion of such applications. We will restrict our discussion to finite outcome spaces. In all examples in which we want to apply the theory, the outcome spaces are finite, and it is probably a good idea to study and develop a new theory in the simplest possible setting first anyway.

The current paper is organized as follows. In Section \ref{sec:basic} we introduce belief functions and some basic results. In Section \ref{sec:conditioning}  and \ref{sec:independence} we develop the backbone of our calculus by discussing the concepts of respectively conditioning and independence. In Section \ref{sec:dempster} we explain why we reject Dempster's rule of combination using the insights of our own calculus. In Section \ref{sec:examples} we discuss the behavior of the theory in a gambling example, and with an example in a forensic setting. After that we discuss the relation between belief functions and a special collection of classical probability distributions in Section \ref{subsec:pbel}, a law of large numbers for belief functions in Section \ref{subsec:frequency}, and finally in Section \ref{subsec:betting} we offer a betting interpretation for belief functions.

Is is unavoidable that in the discussion of the basics of the theory, there is some overlap with our companion paper \cite{kerkmee15}. It has been our aim to make both papers self-contained. 

\section{The basics}
\label{sec:basic}

Let $\Omega$ be a finite outcome space. We want to make statements about the elements of $\Omega$ in the presence of uncertainty. The classical way to do this 
is by means of a suitable probability distribution on $\Omega$. A probability distribution assigns a non-negative weight $p(\omega)$ to each element $\omega \in \Omega$ in such a way that the total weight is equal to 1. We may, for instance, express our uncertainty about who is the culprit by means of such a probability distribution. The probability that the culprit can be found in a subset $A$ of $\Omega$ is then equal to 
\begin{equation}
\label{kansen}
P(A):=\sum_{\omega \in A} p(\omega).
\end{equation}
The probability measure $P$ can be interpreted as epistemic, frequentistic or otherwise, depending on the context and personal 
taste. The weight $p(\omega)$ represents the probability, degree of belief, or confidence in the outcome $\omega$, and $P(A)$ represents our probability, degree of belief, or confidence in an outcome which is contained in $A$. In classical probability theory, a subset of $\Omega$ is also called an {\em event} or a {\em hypothesis}, and $P$ describes the probability of all such events or hypotheses.

Next we define basic belief assignments and belief functions. The difference between a {\em basic belief assignment} and a probability distribution, is that the former assigns weights to nonempty \emph{subsets} of $\Omega$ rather than to individual outcomes. We write $2^{\Omega}$ for the collection of all subsets of $\Omega$.

\begin{definition}
A function $m: 2^{\Omega} \rightarrow [0,1]$ is a {\em basic belief assignment} if $m(\emptyset)=0$ and
\begin{equation}
\label{eq:additivitym}
\sum_{C \subseteq \Omega} m(C)=1.
\end{equation}
\end{definition}
Whereas $p(\omega)$ represents the probability or confidence in the outcome $\omega$, $m(C)$ represents our confidence in an outcome in $C$ without any further specification which element of $C$ is the outcome. In slightly different words, we can interpret $m(C)$ as the probability of having knowledge precisely $C$.

It may appear that there is not much difference between $P$ and $m$, but in fact there is. The crucial difference between $P$ and $m$ is that the weight of a subset $C$ of $\Omega$  is not immediately related to the weights of the elements or subsets of $C$. 
For instance, if we have no clue whatsoever about the outcome, that is, if we have no information at all other than that the outcome is in $\Omega$, then we may express this by putting $m(\Omega)=1$ and $m(A)=0$ for all strict subsets of $\Omega$. If we want to assign belief $1/2$ to the union of $a$ and $b$ without making individual statements about $a$ and $b$, then we can express this as $m(\{a,b\})=1/2$ and simultaneously $m(\{a\})=m(\{b\})=0$.
It is also possible that a basic belief functions only assigns positive weight to singletons. In such a case, we are back in the classical situation. 

The quantity $m(C)$ is sometimes referred to as the {\em weight of the evidence} that points precisely to $C$.
We should view $m$ as the analogue of $p$ in the classical description above. Next we define the analogue of $P$, which is called a {\em belief function}. 

We want to quantify how much belief we can assign to a subset $A$ of 
$\Omega$. To this end, we consider all sets $C$ in $\Omega$ with $C \subseteq  A$, which are precisely the events whose occurrence implies the occurrence of $A$.
The belief in a set $A$ now is the sum of the weights of all subsets of $A$.  In terms of evidence, the belief in $A$ is the total weight of all evidence which implies $A$. 

\begin{definition}
\label{def:belief}
Given a basic belief assignment $m: 2^{\Omega} \rightarrow [0,1]$, the corresponding {\em belief function} $\B: 2^\Omega \rightarrow [0,1]$ is defined by
\begin{equation}
\B(A) := \sum_{C \subseteq A} m(C).
\end{equation}
\end{definition}

The most natural interpretation of the theory is to see a subset of outcomes as a representation or description of individual or shared knowledge, and that we quantify our knowledge with belief functions. In slightly different words, the belief in $A$ is the probability to have information or evidence which implies the occurrence of $A$. 

The set $\Omega$ on the one hand, and singletons on the other, are the extreme states of knowledge, representing respectively ignorance (we do not know anything about the outcome other than that it is in $\Omega$) and complete knowledge (we know which element of $\Omega$ is the outcome). The empty set represents knowing a contradiction which is impossible and hence has probability zero. 

\begin{example} ({\em Probability distributions})
Every probability distribution is a belief function. To see this, let $P: 2^\Omega \rightarrow [0,1]$ be a probability distribution. Set $m(\{ \omega \})=P(\{ \omega\})$ for all $\omega \in \Omega$ and $m(C)=0$ for all $C$ such that $|C|>1$. Then we get
\begin{equation}
\B(A) = \sum_{a \in A} m(\{a\}) = \sum_{a \in A} P(\{a\}) = P(A)
\end{equation}
for every $A \subseteq \Omega$. Probability distributions are belief functions for which the corresponding basic belief assignment only assigns positive weight to singletons.

If $m(C)>0$ for some $C$ with $|C|>1$, then $\B$ is not a probability distribution because it not additive: for any nonempty, disjoint $A,B \subseteq \Omega$ such that $A\cup B = C$ we find
\begin{equation}
\B(A \cup B) > \B(A) + \B(B).
\end{equation}
\end{example}

\begin{example}
\label{exsecond}
Suppose we want to state our beliefs about a suspect being guilty or innocent, so $\Omega = \{$guilty, innocent$\}$ is our outcome space. If the only evidence we have, is evidence of weight $p$ that the suspect is innocent, then we have $m(\{$innocent$\})=p$, $m(\{$guilty$\})=0$ and $m(\Omega)=1-p$. Notice that the belief that the suspect is guilty is not equal to 1 minus the belief that the suspect is innocent. The corresponding belief function $\B$ is given by $\B(\{$guilty$\})=0$, $\B(\{$innocent$\})=p$ and $\B(\Omega)=1$.
\end{example}

\begin{example}
\label{exfirst}
The function $m$ for which $m(\Omega)=1$ and $m(A) =0$ for all other $A \subseteq \Omega$ is a basic belief assignment. The corresponding belief function assigns belief 1 to $\Omega$ and belief zero to all strict subsets of $\Omega$. This belief function expresses total ignorance within a given population $\Omega$, except for the fact that the outcome must be in $\Omega$. As such it addresses the problem noticed in Example \ref{ex2}.
\end{example}

\begin{example}
With reference to the situation described in Example \ref{ex3}, we let $\Omega=\{h,t\}$ be the outcome space of the first croupier, where $h$ stands for `head' and $t$ for `tail'. We set the basic belief assignment $m:2^{\Omega} \rightarrow [0,1]$ by  $m(\{ h \}) = m(\{ t \}) = \frac{1}{2}(1-p)$ and $m(\{h,t\})=p$. 
\end{example}

There is a natural way to identify a belief functions with a collection of probability distributions, namely the collection $\mathcal{P}_{\B}$ of all probability distributions that one can obtain by distributing a total mass of $m(C)$ over the elements of $C$, for all $C \subseteq \Omega$. For instance, if $m(\Omega)=1$, then the corresponding $\mathcal{P}_{\B}$ consists of {\em all} probability distributions on $\Omega$. If $\Omega=\{0,1\}$, and $m(\{0\})=1-m(\{0,1\})=\frac13$, then $\mathcal{P}_{\B}$ consists of all probability distributions that assign probability at least $\frac13$ to $0$. 

It is not difficult to see that
\begin{equation}
\label{identify}
\B(A)= \min_{P \in \mathcal{P}_{\B}} P(A).
\end{equation}
Indeed, only for sets $C$ which are completely contained in $A$, their mass $m(C)$ contributes to $\B(A)$, and those are precisely the sets whose mass cannot be moved outside $A$. This identity might lead to the idea to interpret $\mathcal{P}_\B$ as describing all the possible underlying probability distributions, one of which is the ``correct" one, still keeping the idea that the actual situation is described by an (unknown) probability distribution from the collection $\mathcal{P}_{\B}$. 

It is this interpretation which is criticized in \cite{pearl}. This interpretation of $\mathcal{P}_\B$ is closely tied with the theory of `minimum probability' that has been studied, see e.g.\ \cite{shafer81}. This theory, however, is distinct from our theory. In Section \ref{sec:conditioning}, for example, we will see that conditional belief is \emph{not} equal to the infimum over all conditional probabilities in $\mathcal{P}_\B$. In Section \ref{subsec:pbel} we discuss the interpretation of $\mathcal{P}_\B$ in more detail. 

Whereas belief functions are in general not additive, it follows directly from the definition that they are superadditive, i.e.
\begin{equation}
\B(A \cup B) \geq \B(A) + \B(B)
\end{equation}
for all disjoint $A,B \subseteq \Omega$. The following theorem by Shafer \cite{shafer76} shows that belief functions are characterized by a property between additivity and super-additivity.

\begin{theorem}
\label{thm:beliefalternate}
A function $\B: 2^\Omega \rightarrow [0,1]$ is a belief function if and only if
\begin{itemize}
\item[(B1)] $\B(\Omega)=1$
\item[(B2)] For all $A,B \subseteq \Omega$ we have
\begin{equation}
\label{eq:beliefadditivity}
\B(A \cup B) \geq \B(A) + \B(B) - \B(A \cap B).
\end{equation}
\end{itemize}
\end{theorem}

\begin{remark}
To see that (B2) is stricter than super-additivity, consider the following example. Let $\Omega=\{a,b,c\}$ and set $f(\Omega)=f(\{a,b\})=f(\{b,c\})=1$, $f(\{b\})=\frac{1}{2}$ and $f(C)=0$ for all other $C$. It is easy to check $f$ is superadditive, but (B2) does not hold since
\begin{equation}
1 = f(\Omega) \not\geq f(\{a,b\})+f(\{b,c\})-f(\{b\}) = \frac{3}{2}.
\end{equation}
\end{remark}

Theorem \ref{thm:beliefalternate} can be used to give an alternative definition of belief functions without deriving them from basic belief assignments. Using the theorem we can directly check whether or not a function $f: 2^\Omega \rightarrow [0,1]$ is a belief function. We can use the following lemma by Shafer \cite{shafer76} to retrieve the basic belief assignment corresponding to a given belief function.

\begin{lemma}
Every belief function $\B: 2^\Omega \rightarrow [0,1]$ has a unique corresponding basic probability assignment $m: 2^{\Omega} \rightarrow [0,1]$ which is given by
\begin{equation}
m(A) = \sum_{C \subseteq A} (-1)^{|A|-|C|} \B(C).
\end{equation}
\end{lemma}

\section{Conditioning}
\label{sec:conditioning}

We have discussed the mathematical definitions and basic properties of belief functions. The next thing on the agenda is to investigate how belief functions change when additional or new information is provided. This is akin to the classical situation in which a prior probability is updated into a posterior one, based on additional information. In this section we explain how this works in our setting. The first thing to do is to determine how a belief function changes under additional information, or under a certain hypothesis. This means we need to understand how conditioning works in our context.

The rule we propose for conditioning is described as follows. Suppose we have a basic belief assignment $m$ and corresponding belief function $\B$. We want to condition on an event $H$. The weight of the evidence $m(A)$ for $A$ now becomes weight of evidence for $A \cap H$, if $A$ is consistent with $H$ in the sense that $A \cap H \not=\emptyset$. If $A \cap H =\emptyset$, then the new weight of evidence for $A$ becomes zero. Next we rescale the weights of the evidence in such a way that the weights again sum up to $1$. This can of course only be done if there is evidence with positive weight that is consistent with $H$. This leads to the following definition.

\begin{definition}
\label{def:conditionalm}
Let $m: 2^\Omega \rightarrow [0,1]$ be a basic belief assignment and $\B$ the corresponding belief function. For $H \subseteq \Omega$ such that $\B(H^c) \not=1$  we define the conditional basic belief assignment $m_H: 2^\Omega \rightarrow [0,1]$ by
\begin{equation}
\label{eq:mcond}
m_H(A) := \frac{\sum_{B \cap H = A} m(B)}{1 - \sum_{B \cap H = \emptyset} m(B)},
\end{equation}
for $A \not= \emptyset$ and $m_H(\emptyset)=0$.
\end{definition}

The corresponding conditional belief function $\B_H$ can now be obtained in the obvious way from the basic belief assignment $m_H$, as follows:
\begin{equation}
\label{conditionalbelief}
\begin{aligned}
\B_H(A) & =  \sum_{B \subseteq A } m_H(B) \\
& = \frac{ \sum_{\emptyset \not= C \cap H \subseteq A} m(C)} { 1 - \sum_{C \cap H = \emptyset} m(C) } \\
& = \frac{ \sum_{C \subseteq A\cup H^c} m(C) - \sum_{C \subseteq H^c} m(C)} { 1 - \sum_{C \subseteq H^c} m(C) } \\
& =  \frac{ \B(A \cup H^c) - \B(H^c)}{1-\B(H^c)}
\end{aligned}
\end{equation}
for all $A$ and $H$ such that $\B(H^c) \not=1$.

Readers familiar with the work of Shafer \cite{shafer76} will notice that (\ref{conditionalbelief}) is the same formula as (3.8) in \cite{shafer76}), and is sometimes called Dempster-conditioning. This name is somewhat misleading though. Indeed, Shafer derives the formula as a special case of Dempster's rule, a rule which we reject. It so happens that one can derive (\ref{conditionalbelief}) without Dempster's rule, only making use of our definition of conditional belief. 

Interpreting $m(A)$ as the probability of having knowledge $A$ naturally leads to a probability distribution $P$ on the collections of subsets of $\Omega$. That is, for $\mathcal{A} \subseteq 2^\Omega$ we write
\begin{equation}
\label{kansmaat}
P(\mathcal{A}) = \sum_{A \in \mathcal{A}} m(A).
\end{equation} 
The belief in $A$ now is the probability that $A$ is implied, that is
\begin{equation}
\label{eq:epist}
\B(A)=P(\{C \in 2^\Omega \;:\; C \subseteq A\}).
\end{equation}

We can now express (\ref{def:conditionalm}) in terms of $P$ by writing for $A \not= \emptyset$
\begin{equation}
\label{eq:epistemicm}
\begin{aligned}
m_H(A) & = \frac{\sum_{C \cap H = A} m(C)}{ \sum_{C \cap H \not= \emptyset} m(C)} \\
& = \frac{P(\{ C \subseteq \Omega \;:\; C \cap H = A \})}{P(\{C \subseteq \Omega \;:\; C \cap H \not= \emptyset \})} \\
& = P(\{ C \subseteq \Omega \;:\; C \cap H = A \} \; | \; \{C \subseteq \Omega \;:\; C \cap H \not= \emptyset \}).
\end{aligned}
\end{equation}
Conditional belief can be expressed as
\begin{equation}
\label{eq:epistemicbel}
\begin{aligned}
\B_H(A) & = \sum_{B \subseteq A } m_H(B) \\
& = \sum_{B \subseteq A } P(\{ C \subseteq \Omega \;:\; C \cap H = B \} \; | \; \{C \subseteq \Omega \;:\; C \cap H \not= \emptyset \}) \\
& = P(\{ C \subseteq \Omega \;:\; C \cap H \subseteq A \} \; | \; \{C \subseteq \Omega \;:\; C \cap H \not= \emptyset \}).
\end{aligned}
\end{equation}
In words, (\ref{eq:epistemicm}) and (\ref{eq:epistemicbel}) show that our notion of conditioning can be seen as classically conditioning $P$ on the collection of outcomes that are consistent with $H$, and then lumping all outcomes that are the same under $H$ together.

\begin{example}
In the special case that $\B=P$ is a probability distribution, the notion of conditional belief coincides with the notion of conditional probability, i.e.
\begin{equation}
\B_H(A) = \frac{\sum_{\omega \in A \cap H} m(\{\omega\}) }{\sum_{\omega \in H} m(\{\omega\})} = P(A|H),
\end{equation}
for every $A \subseteq \Omega$ and $H$ such that $1-\B(H^c)=P(H)>0$.
\end{example}

\begin{example}
\label{ex:conditioning1}
Suppose we have a case in which the suspects are two parents and their son, so $\Omega=\{\mathrm{Father},\mathrm{Mother},\mathrm{Son}\}$. We have a lot of evidence that points to the parents, none of which points to one of them in particular. Further, we have some evidence that points to the son. The corresponding basic belief assignment is, say, $m(\{\mathrm{Father},\mathrm{Mother}\})=\frac{9}{10}$ and $m(\{\mathrm{Son}\})=\frac{1}{10}$. Under the hypothesis $H$ that it is a man, i.e.\ $H=\{\mathrm{Father},\mathrm{Son}\}$, the evidence against the parents counts as evidence against the father, so
\begin{equation}
\label{condi}
 m_H(\{ \mathrm{Father} \}) = \frac{9}{10}.
\end{equation}
\end{example}

The next example shows that while (\ref{identify}) holds, there are $\B$ and $A,H \subseteq \Omega$ such that
\begin{equation}
\label{eFH}
\B_H(A) \not = \inf \{ P(A|H) \;:\; P \in \mathcal{P}_\B \}.
\end{equation}
The right hand side in (\ref{eFH}) goes under the name FH-conditioning, after Fagin and Halpern \cite{FH}. Hence, the example shows that Dempster-conditioning, and FH-conditioning need not lead to the same result. 

\begin{example} ({\em Continuation of Example \ref{ex:conditioning1}}.)
\label{ex:conditioning2}
The collection of probability distributions that we can obtain by distributing weight on $\{$Father,Mother$\}$ over $\{$Father$\}$ and $\{$Mother$\}$ is 
\begin{equation}
\mathcal{P}_\B = \left\{ P_c \;:\; 0 \leq c \leq \frac{9}{10} \right\},
\end{equation}
where the probability distribution $P_c: 2^\Omega \rightarrow [0,1]$ is given by
\begin{equation}
\begin{aligned}
 P_c(\{ \mathrm{Father} \}) & :=  c, \\
 P_c(\{ \mathrm{Mother} \}) & :=  \frac{9}{10} - c, \\
 P_c(\{ \mathrm{Son} \}) & :=  \frac{1}{10}.
 \end{aligned}
\end{equation}
Since
\begin{equation}
P_c(\{ \mathrm{Father} \}|H) = \frac{c}{c+\frac{1}{10}},
\end{equation}
we find
\begin{equation}
\label{condiklassiek}
\inf \{P(\{ \mathrm{Father} \}|H) \;:\; P \in \mathcal{P}_\B \} = 0 .
\end{equation}
Compared to (\ref{condi}), we think the answer in (\ref{condi}) seems more appropriate than the one in 
(\ref{condiklassiek}).
\end{example}

Although our notion of conditioning generalizes the classical notion, there  are significant differences between our conditioning and the classical one. To this end, consider the following instructive example which we subsequently discuss, and which appears in a slightly different formulation also in \cite{pearl}.

\begin{example}
\label{pasop}
Suppose $\Omega = \{0,1\}^2$, and write an element of $\Omega$ as $(x,y)$. Consider the following basic belief assignment $m$ on $\Omega$:
\begin{equation}
\begin{aligned}
m \left( \; \begin{tabular}{ | l | c | c | c |}
    \hline
 & $x=0$ & $x=1$ \\ \hline
$y=0$ &  $\ast$   &     \\ \hline
$y=1$ &  $\ast$  &       \\ \hline
\end{tabular} \; \right) = \frac{1}{2}
\end{aligned}
\end{equation}
and
\begin{equation}
\begin{aligned}
m \left( \; \begin{tabular}{ | l | c | c | c |}
    \hline
 & $x=0$ & $x=1$ \\ \hline
$y=0$ &    & $*$    \\ \hline
$y=1$ &     &  $*$        \\ \hline
\end{tabular}  \; \right) = \frac{1}{2}.
\end{aligned}
\end{equation}
Following the rules of our calculus, it is not difficult to see that $\B(x=y)=0$.
However, this outcome may not be so intuitive at first sight, which becomes apparent when we first condition on the outcome of $y$. Indeed, we have that $\B_{y=0}(x=y)=\B_{y=1}(x=y)=\frac12$ while at the same time $\B(x=y)=0$. This phenomenon deserves a discussion. 
\end{example}

We can gain understanding of the paradox in Example \ref{pasop} by looking at the law of total probability from classical probability calculus:
\begin{equation}
\label{totalprob}
P(A)= P(A|B^c)P(B^c) + P(A|B)P(B),
\end{equation}
for all events $A$ and $B$ with $P(B)>0$. This law, combined with the fact that $P(B)+P(B^c)=1$, gives the very intuitive result that if 
$P(A|B)=P(A|B^c)=\alpha$, say, then it follows that $P(A)$ must also be equal to $\alpha$. (In some references this phenomenon is called the {\em sandwich principle}, see e.g.\ \cite{pearl} and references therein.) However, the analogue for belief functions, that is, if $\B_B(A)=\B_{B^c}(A)=\alpha$, then it follows that $\B(A)$ must also be equal to $\alpha$, does not hold in general, as Example \ref{pasop} illustrates. We now explain why this does not hold in Example \ref{pasop}.

First we note that $\B_{y=0}(x=y)=\B_{y=1}(x=y)=\frac12$ is not at all controversial: if we simply know the value of $y$, then all uncertainty that is left, is that of a fair coin flip. The paradox arises, because at the same time we have $\B(x=y)=0$. This zero belief is explained by the fact that we do not know how the outcome of $y$ is produced. It may in fact be the case that for some reason we do not know, the outcome of $y$ is always the opposite of $x$. Therefore, our belief in $x=y$ should indeed be zero. Based on $\B_{y=0}(x=y)$ and $\B_{y=1}(x=y)$, for which how $y$ is produced is completely irrelevant, we can not infer anything about $\B(x=y)$. 

It would be an entirely different matter to condition on the outcome of $x$. We know how the outcome of $x$ is produced, namely as the result of a fair coin flip. And once we know the outcome of $x$, we are still ignorant about the outcome of $y$, i.e. $\B_{x=0}(x=y)=\B_{x=1}(x=y)=0$ and therefore it is completely reasonable to conclude that $\B(x=y)=0$. So when we condition on $x$, then the paradox does not arise. This is a special case of a more general situation in which an analogue of the law of total probability does hold, which we express in the following lemma. The proof follows immediately from definitions.

\begin{lemma}
Let $B_1,\ldots,B_n \subseteq \Omega$ be a partition of $\Omega$ such that $\B(B_i)>0$ for every $i$ and for every $C$ with $m(C)>0$ we have $C \subseteq B_i$ for some $i$. Then
\begin{equation}
\label{eq:totbel}
\sum_{i=1}^n \B(B_i) = 1 \quad \mathrm{and}\quad \B(A) = \sum_{i=1}^n \B(B_i) \B_{B_i}(A)
\end{equation}
for all $A \subseteq \Omega$.
\end{lemma}

Notice that in Example \ref{pasop} with $B_1=\{y=0\}$ and $B_2=\{y=1\}$, the condition of the lemma is, as expected, not satisfied. For $A=\{x=y\}$, the second part of (\ref{eq:totbel}) does hold, because we do have that
$$
\B(x=y)=\B_{y=0}(x=y)\B(y=0) + \B_{y=1}(x=y) \B(y=1),
$$
since both sides are equal to zero. However, the second part of (\ref{eq:totbel}) does not always hold, since for $A=\{x=0\}$ we find
$$
\B(x=0) \not= \B_{y=0}(x=0)\B(y=0) + \B_{y=1}(x=0) \B(y=1).
$$

\section{Independence}
\label{sec:independence}

Now that we have a notion of conditioning, we can also introduce a notion of independence, and to this end we consider the following situation. Let $\Omega_1$ and $\Omega_2$ be outcome spaces of two phenomena that we would describe as `independent'. Now we consider this two phenomena simultaneously. We set $\Omega := \Omega_1 \times \Omega_2$ and let $X: \Omega \rightarrow \Omega_1$ and $Y: \Omega \rightarrow \Omega_2$ be the projections onto respectively the first and second coordinate. On $\Omega$ we define a basic belief assignment $m: 2^\Omega \rightarrow [0,1]$ and corresponding belief function $\B$. We want $m$ and $\B$ to reflect the `independent' nature of the two phenomena. To do that, we need a mathematical definition of independence that is consistent with our intuitive idea about  `independence'. 

There are now at least three natural ways to proceed, and we explore them now, together with their relationships.
 
In the first approach we take conditional beliefs as the starting point, and require that
\begin{equation}
\label{ind1}
\B_{Y \in B}(X \in A) = \B(X \in A)
\end{equation}
and
\begin{equation}
\label{ind2}
\B_{X \in A}(Y \in B) = \B(Y \in B),
\end{equation}
for all $A\subseteq \Omega_1$ and $B \subseteq \Omega_2$ for which the conditional beliefs are defined. It is not difficult to show directly that (\ref{ind1}) and (\ref{ind2}) are equivalent, but this will follow also from Theorem \ref{stellingonaf} below, so we do not prove this here. 

In the second approach, we proceed via a product form for the belief $\B(X \in A; Y \in B)$. As a natural generalization of independence in probability theory, we require that this product is equal to the product of the `marginal' beliefs, i.e.
\begin{equation}
\label{appr2}
\B(X \in A; Y \in B) = \B(X \in A)\B(Y \in B),
\end{equation}
for all $A \subseteq \Omega_1$ and $B \subseteq \Omega_2$. 

Instead of looking at a product form for $\B(X \in A; Y \in B)$, we can also look at a product form for $m(X \in A; Y \in B)$, which is our third approach.  We first define $\B_1$ and $\B_2$ to be the `marginal' belief functions, i.e.
\begin{equation}
\B_1(A) := \B(X \in A) \;\; \mathrm{and} \;\; \B_2(B) := \B(Y \in B),
\end{equation}
for all $A \subseteq \Omega_1$ and $B \subseteq \Omega_2$. 
It follows from Theorem \ref{thm:beliefalternate} that $\B_1$ and $\B_2$ are indeed belief functions.
Let $m_1$ and $m_2$ the corresponding `marginal' basic belief assignments of respectively $\B_1$ and $\B_2$. Since
\begin{equation}
\begin{aligned}
\B_1(A) & =  \sum_{B \subseteq A \times \Omega_2} m(B) \\
& =  \sum_{B |\; X(B) \subseteq A} m(B) \\
& = \sum_{B \subseteq A} \;\; \sum_{C|\; X(C)=B} m(C),
\end{aligned}
\end{equation}
it follows that
\begin{equation}
\label{eq:defmarginalm}
m_1(A) = \sum_{C \; |\; X(C)=A} m(C).
\end{equation}
A similar expression of course is true for $m_2$. Notice that $m_1(A)$ is in general \emph{not} the same as $m(X \in A)$, since the set $\{X \in A\}$ is in general not the only set with positive basic belief which projects onto $A$. Our third approach to independence uses classical independence to require that
\begin{equation}
\label{ind3}
m(X \in A; Y \in B) = m_1(A)m_2(B)
\end{equation}
for all $A \subseteq \Omega_1$ and $B \subseteq \Omega_2$. 

Now we investigate the relations between the three approaches. The next two examples show that the requirements of respectively the first and second approach are weaker than the requirement of the third. 

\begin{example}
Suppose $\Omega_1=\Omega_2=\{0,1\}$ and
\begin{equation}
\begin{aligned}
m \left( \; \begin{tabular}{ | l | c | c | c |}
    \hline
 & X=0 & X=1 \\ \hline
Y=0 &  $\ast$   & $\ast$    \\ \hline
Y=1 &  $\ast$  & $\ast$      \\ \hline
\end{tabular} \; \right) = \frac{1}{4},  \\
m \left( \; \begin{tabular}{ | l | c | c | c |}
    \hline
 & X=0 & X=1 \\ \hline
Y=0 &  $\ast$  &     \\ \hline
Y=1 &     &          \\ \hline
\end{tabular}  \; \right) = \frac{1}{2}, \\
m \left( \;  \begin{tabular}{ | l | c | c | c |}
    \hline
 & X=0 & X=1 \\ \hline
Y=0 &  $\ast$   & $\ast$    \\ \hline
Y=1 &  $\ast$  &          \\ \hline
\end{tabular} \; \right) = \frac{1}{4}.  \\
\end{aligned}
\end{equation}
Then
\begin{equation}
\begin{aligned}
\B_{Y=1}(X=1)= \B_{Y=0}(X=1)= \B(X=1) & = 0, \\
\B_{Y=1}(X=0)= \B_{Y=0}(X=0)= \B(X=0) & = \frac{1}{2},
\end{aligned}
\end{equation}
so the requirement of the first approach is satisfied. But
\begin{equation}
\frac{1}{2} = m(X=0; Y =0) \not= m_1(\{0\})m_2(\{0\}) = \frac{1}{2} \cdot \frac{1}{2} = \frac{1}{4}.
\end{equation}
\end{example}

\begin{example}
Suppose $\Omega_1=\Omega_2=\{0,1\}$ and
\begin{equation}
\begin{aligned}
m \left( \;  \begin{tabular}{ | l | c | c | c |}
    \hline
 & X=0 & X=1 \\ \hline
Y=0 &  $\ast$  &    \\ \hline
Y=1 &   &  $\ast$       \\ \hline
\end{tabular} \; \right) = 1.
\end{aligned}
\end{equation}
We have
\begin{equation}
\B(X \in A; Y \in B) = \B(X \in A)\B(Y \in B)
\end{equation}
trivially for all $A,B \subseteq \{0,1\}$, but
\begin{equation}
0 = m(X \in \{0,1\}; Y \in \{0,1\}) \not= m_1(\{0,1\})m_2(\{0,1\}) = 1.
\end{equation}
\end{example}

As suggested by the examples, the problem is that there is positive mass on sets that are not `rectangular', i.e. $S \subseteq \Omega$ such that $S \not= X(S) \times Y(S)$. 

Consider a $C \subseteq \Omega$ with $m(C)>0$. The set $X(C) \subseteq \Omega_1$ gives the outcomes of the first phenomenon that $C$ is consistent with. If we condition on $\{Y=y\}$ for some $y \in Y(C)$, evidence for $C$ will become evidence for $C \cap \{Y=y\}$. Since we want to model the two phenomena as independent, it is reasonable to ask that conditioning on $\{Y=y\}$ changes nothing about the outcomes of the first phenomenon that this individual piece of evidence is consistent with. So
\begin{equation}
\label{eq:projectionofevidence}
\forall y \in Y(C) \;\;\; X(C \cap \{Y = y\}) = X(C).
\end{equation}
It follows directly that (\ref{eq:projectionofevidence}) holds if and only if $C= X(C) \times Y(C)$. If $m(C)>0$ only if $C= X(C) \times Y(C)$, we say that $m$ concentrates on rectangles. Adding this constraint to the requirements of the first two approaches, makes all three approaches equivalent.

\begin{theorem}
\label{stellingonaf}
The following statements are equivalent:
\begin{enumerate}
\item[(1)] $m$ concentrates on rectangles and
\begin{equation}
\B_{Y \in B}(X \in A) = \B(X \in A)
\end{equation}
for all $A \subseteq \Omega_1$ and $B \subseteq \Omega_2$ with $\B(Y \in B^c) \not=1$,
\item[(2)] $m$ concentrates on rectangles and
\begin{equation}
\B(X \in A; Y \in B) = \B(X \in A)\B(Y \in B)
\end{equation}
for all  $A \subseteq \Omega_1$ and $B \subseteq \Omega_2$,
\item[(3)] we have
\begin{equation}
m(X \in A; Y \in B) = m_1(A)m_2(B)
\end{equation}
for all  $A \subseteq \Omega_1$ and $B \subseteq \Omega_2$.
\end{enumerate}
\end{theorem}

\begin{proof}
We prove (1)$\Rightarrow$(2), (2)$\Rightarrow$(3) and (3)$\Rightarrow$(1).

We start with (1)$\Rightarrow$(2), so assume (1) holds. First we note that in general it holds that $\B(A\cup B)= \B(A) + \B(B) - \B(A \cap B)$ if for every $C \subseteq A \cup B$ with $m(C)>0$ we have either $C \subseteq A$ or $C \subseteq B$. Because $m$ concentrates on rectangles, we have that for every $C \subseteq \{X \in A \vee Y \in B\}$ with $m(C)>0$ that $C \subseteq \{X \in A\}$ or $C \subseteq \{Y \in B\}$. So
\begin{equation}
\label{eq:vee}
\B(X \in A \vee Y \in B) = \B(X \in A) + \B(Y \in B) - \B(X \in A; Y \in B)
\end{equation}
for every $A \subseteq \Omega_1$ and $B \subseteq \Omega_2$.

Now let $A \subseteq \Omega_1$ and $B \subseteq \Omega_2$ such that $B(Y \in B) \not=1$. We find using respectively (1), (\ref{conditionalbelief}) and (\ref{eq:vee}), that
\begin{equation}
\begin{aligned}
\B(X \in A) & = \B_{Y \in B^c}(X \in A) \\
& = \frac{\B(X \in A \vee Y \in B) - \B(Y \in B) }{1 - \B(Y \in B)} \\
& = \frac{ \B(X \in A) - \B(X \in A; Y \in B) }{1 - \B(Y \in B)}.
\end{aligned}
\end{equation}
Rewriting this equation gives
\begin{equation}
\label{eq:statement2}
\B(X \in A; Y \in B) = \B(X \in A)\B(Y \in B).
\end{equation}
Now we note that (\ref{eq:statement2}) is trivially true for $A \subseteq \Omega_1$ and $B \subseteq \Omega_2$ with $\B(Y \in B)=1$, so (2) holds.

We continue with (2)$\Rightarrow$(3), so assume (2) holds. We prove (3) by induction on $|A|+|B|$. Clearly, if $|A|+|B|=0$ for $A \subseteq \Omega_1$ and $B \subseteq \Omega_2$, then $m(X \in A; Y \in B)= 0 =  m_1(A)m_2(B)$. Now suppose that $m(X \in A'; Y \in B')= m_1(A')m_2(B')$ for all  $A' \subseteq \Omega_1$ and $B' \subseteq \Omega_2$ with $|A'|+|B'|<n$. Let $A \subseteq \Omega_1$ and $B \subseteq \Omega_2$ be  such that $|A|+|B|=n$. Because of (2) we have
\begin{equation}
\label{eq:proof1}
\begin{aligned}
& m(X \in A; Y \in B) \\
& = \B(X \in A; Y \in B) - \sum_{ \substack{ A' \subseteq A, B' \subseteq B \\ |A'|+|B'|<n} } m(X \in A'; Y \in B') \\
& = \B(X \in A)\B(Y \in B) - \sum_{ \substack{ A' \subseteq A, B' \subseteq B \\ |A'|+|B'|<n} } m(X \in A'; Y \in B') \\
\end{aligned}
\end{equation}
Our induction hypothesis gives
\begin{equation}
\label{eq:proof2}
\begin{aligned}
& \sum_{ \substack{ A' \subseteq A, B' \subseteq B \\ |A'|+|B'|<n} } m(X \in A'; Y \in B') \\
& = \sum_{ \substack{ A' \subseteq A, B' \subseteq B \\ |A'|+|B'|<n} } m_1(A')m_2(B') \\
& = \sum_{B \subsetneq B'} m_1(A)m_2(B') + \sum_{A' \subsetneq A} \sum_{B' \subseteq B} m_1(A')m_2(B') \\
& = m_1(A)(\B(Y\in B)-m_2(B)) + \B(Y \in B)( \B(X \in A) - m_1(A) ) \\
& = - m_1(A)m_2(B) + \B(Y \in B)\B(X \in A). \\
\end{aligned}
\end{equation}
Combining (\ref{eq:proof1}) and (\ref{eq:proof2}) gives us $m(X \in A; Y\in B)=m_1(A)m_2(B)$. So (3) holds.

The only implication left to show is (3)$\Rightarrow$(1), so assume (3) holds. Because $m_1$ and $m_2$ are basic belief assignments on respectively $\Omega_1$ and $\Omega_2$ we find
\begin{equation}
\sum_{A \subseteq \Omega_1} \sum_{B \subseteq \Omega_2} m(X \in A; Y\in B) = \sum_{A \subseteq \Omega_1} m_1(A) \sum_{B \subseteq \Omega_2} m_2(B) = 1.
\end{equation}
This implies $m(C)=0$ if $C \not= \{X \in A; Y \in B\}$ for some $A \subseteq \Omega_1$ and $B \in \Omega_2$. Hence $m$ concentrates on rectangles. Further, for $A \subseteq \Omega_1$ and $B \subseteq \Omega_2$ such that $\B(Y \in B^c)\not=1$, we have
\begin{equation}
\begin{aligned}
\B_{Y \in B}(X \in A) & = \frac{ \sum_{A' \subseteq A} \sum_{B' \subseteq B} m(X \in A'; Y \in B')} { 1 - \sum_{A' \subseteq \Omega_1} \sum_{B' \subseteq B^c } m(X \in A'; Y \in B') } \\
& = \frac{ \sum_{A' \subseteq A} m_1(A') \sum_{B' \subseteq B} m_2(B')} { 1 - \sum_{A' \subseteq \Omega_1} m_1(A') \sum_{B' \subseteq B^c } m_2(B') } \\
& = \frac{ \sum_{A' \subseteq A} m_1(A') \sum_{B' \subseteq B} m_2(B')} { 1 - \sum_{B' \subseteq B^c } m_2(B') } \\
& = \frac{ \sum_{A' \subseteq A} m_1(A') \sum_{B' \subseteq B} m_2(B')} { \sum_{B' \subseteq B } m_2(B') } \\
& = \sum_{A' \subseteq A} m_1(A') = \B(X \in A).
\end{aligned}
\end{equation}
So (1) holds.
\end{proof}

We take the requirement of the third approach to be our actual definition of independence.

\begin{definition}
\label{def:independence}
The projections $X$ and $Y$ are independent if for all $A \subseteq \Omega_1$ and $B \subseteq \Omega_2$ we have
$$
m(X \in A; Y \in B)=m_1(A)m_2(B).
$$
\end{definition}

Our notion of independence generalizes the classical notion. In fact, in the classical situation $m$ concentrates on singletons, which of course are all rectangles and hence the three approaches of independence are the same in that case.

To see how we can interpret independence, we write the $P$ from (\ref{kansmaat}) for the probability distribution corresponding to $m$ on $\Omega_1 \times \Omega_2$ and we write $P_1$ and $P_2$ for the probability distributions corresponding to respectively $m_1$ and $m_2$. It follows directly from the definitions that 
\begin{equation}
\label{eq:episindep}
P( \{ \{X \in A; Y \in B\} \;:\; A \in \mathcal{A}, B \in \mathcal{B} \} ) = P_1(\mathcal{A})P_2(\mathcal{B})
\end{equation}
for all  $\mathcal{A} \subseteq 2^{\Omega_1}$ and $\mathcal{B} \subseteq 2^{\Omega_2}$, is equivalent with  Definition \ref{def:independence}. So (\ref{eq:episindep}) gives an interpretation of our notion of independence: for any $\mathcal{A} \subseteq 2^{\Omega_1}$ and $\mathcal{B} \subseteq 2^{\Omega_2}$,  the probability we have an outcome for $X$ in $\mathcal{A}$ and an outcome for $Y$ in $\mathcal{B}$ is the product of the individual probabilities.

Now that we have introduced independence, we revisit Example \ref{pasop}.

\begin{example}
Consider the situation in Example \ref{pasop} again. The basic belief assignment described there arises for instance when $X$ denotes the outcome of someone flipping a fair coin, and we have no information whatsoever about the way the outcome of $Y$ and how it is produced. Indeed, the marginal belief function of $X$ simply assigns mass $\frac12$ to both outcomes, whereas the marginal of $Y$ represents complete ignorance. 

It is easy to check that $X$ and $Y$ are independent. In fact, if we want the marginal belief functions of $X$ and $Y$ to be as given, $X$ and $Y$ are necessarily independent. This may sound strange, but notice that in our theory $Y$ is `degenerate', since its marginal belief function concentrates on $\{0,1\}$. In the classical theory, it is also the case that a degenerate random variable is independent of any other random variable on the same probability space.
\end{example}

\section{Rejection of Dempster's rule of combination}
\label{sec:dempster}

As we have mentioned in the introduction, we do not need Dempster's rule of combination. Nevertheless we want to spend some lines on it, because this rule has been the most important reason to reject the theory of belief function in the past. Now that we have developed our notions of conditioning and independence, we can explain why we think Dempster's rule deserves no place in our theory of belief functions.

We start by introducing the rule. Suppose $m_1$ and $m_2$ are two basic belief assignments on the same space $\Omega$. Dempster's rule of combination states that if $m_1$ and $m_2$ are `based on independent evidence', we can define a canonical basic belief assignment $m_1 \oplus m_2$ on $\Omega$ `combining' $m_1$ and $m_2$, and which is given by
\begin{equation}
\label{eq:introdempster}
(m_1 \oplus m_2)(A) = \frac{ \sum_{B,C | B \cap C = A} m_1(B)m_2(C) }{ 1 - \sum_{B,C | B \cap C = \emptyset} m_1(B)m_2(C)  }.
\end{equation}

It is easy to check that $m_1 \oplus m_2$ is indeed a basic belief assignment on $\Omega$, but we think  this basic belief assignment is in general not in any way a meaningful `combination' of  $m_1$ and $m_2$. To see why, we first use our own theory of belief functions to interpret (\ref{eq:introdempster}). We define a basic belief assignment $m$ on the product space $\Omega^2$ by treating $m_1$ and $m_2$ as basic belief assignments corresponding to independent phenomena, using Definition \ref{def:independence}, i.e.
\begin{equation}
m(A \times B) := m_1(A)m_2(B)
\end{equation}
for $A,B \subseteq \Omega$.
Let $H := \{ (\omega,\omega) \;:\; \omega \in \Omega \} \subseteq \Omega^2$, which is the set for which the outcomes are identical. Then for nonempty $A \subseteq \Omega$ we find
\begin{equation}
\label{eq:dempster}
\begin{aligned}
m_H(\{ (\omega,\omega) \;:\; \omega \in A \}) & = \frac{ \sum_{D \cap H = \{ (\omega,\omega) \;:\; \omega \in A \}} m(D) }{ 1 - \sum_{D \cap H = \emptyset} m(D) } \\
& = \frac{ \sum_{D_1,D_2 \;|\; D_1 \cap D_2 = A} m_1(D_1)m_2(D_2) }{ 1 - \sum_{D_1,D_2 \;|\; D_1 \cap D_2 = \emptyset} m_1(D_1)m_2(D_2) } \\
& = (m_1 \oplus m_2)(A).
\end{aligned}
\end{equation}

Equation \ref{eq:dempster} shows that if $m_1$ and $m_2$ describe independent phenomena, $m_1 \oplus m_2$ is the basic belief assignment after we learned that they had the same outcome. This is curious because $m_1$ and $m_2$ are concerning the \emph{same} phenomenon, which is the very opposite of $m_1$ and $m_2$ describing independent phenomena. We think that this is the heart of the problem of Dempster's rule: it does some computation treating $m_1$ and $m_2$ as corresponding to independent phenomena, while the description claims it does some computation treating $m_1$ and $m_2$ as corresponding to the same phenomenon. The confusion between `independent evidence' and independent phenomena may find its origin in the problematic nature of the notion of `independent evidence': if evidence concerns the same phenomenon, then automatically the all evidence is dependent because all the evidence depends on the true outcome of the phenomenon. Whatever one precisely means with `independent evidence', however, it is clear that it should not be the same as evidence concerning independent phenomena. This leads to absurd results, like the following example illustrates.

\begin{example}
Suppose we are going to flip a coin and set $\Omega=\{h,t\}$ for the outcomes, where $h$ stands for head and $t$ for tail. Previous flips of the coin have given us the information that head comes up 60\% of times, hence $m_1(\{h\})=\frac{3}{5}$ and $m_1(\{t\})=\frac{2}{5}$. A second investigation based on the shape of the coin leads to the same conclusion, so $m_2=m_1$. Then combining these two basic belief assignments using Dempster's rule gives
\begin{equation}
(m_1 \oplus m_2)(\{h\}) = \frac{\left(\frac{3}{5}\right)^2} {\left(\frac{3}{5}\right)^2 + \left(\frac{2}{5}\right)^2}  = \frac{9}{13} \approx 0.69
\end{equation}
and
\begin{equation}
(m_1 \oplus m_2)(\{t\}) = \frac{\left(\frac{2}{5}\right)^2} {\left(\frac{3}{5}\right)^2 + \left(\frac{2}{5}\right)^2}  = \frac{4}{13} \approx 0.31.
\end{equation}
Hence Dempster's rule leads to the conclusion that the coin is much more biased than both of our two sources of evidence agreed on. Although one could try to argue that confirmation could lead to more belief, it should at least not lead to \emph{less} belief. This is of course exactly what happens with our belief in the outcome `tail': while $m_1(\{t\}=m_2(\{t\}=\frac{2}{5}$, we have $(m_1 \oplus m_2)(\{t\}) = \frac{4}{13} < \frac{2}{5}$. From this it follows that the only acceptable outcome of `combining' $m_1$ and $m_2=m_1$ would be $m_1$ again and not $m_1 \oplus m_2$.
\end{example}

We reject Dempster's rule and do not see any need for a rule that `combines' basic belief assignments on the same space. Since the theory of belief functions is a generalization of probability theory, we think it makes sense to generalize concepts like conditioning and independence. Probability theory, however, does not have a canonical rule for `combining' probability distributions on the same outcome space into a new canonical probability distribution. Both probability theory and our theory of belief functions, as applied to forensic cases in our companion article 
\cite{kerkmee15} seem to function fine without such a rule. Furthermore, the fact that in the well studied case of probability distributions such a rule is not known, gives us reason to be skeptical about being able to construct a plausible rule for the general case.

\section{Examples}
\label{sec:examples}

We start with an example from the forensic setting, and show how our theory can avoid the problem of having to choose a prior, producing a less arbitrary result. This is a typical case in which we think our theory should be used, and we once more refer to our companion paper \cite{kerkmee15} for further forensic examples. After that, we also discuss a more traditional gambling example.

\subsection{A forensic example}
\label{subsec:forensic}

Suppose we know that one of two persons, say $A$ and $B$, has committed a certain crime. We write $E$ for the event that the person that committed the crime and $A$ both have a certain (DNA) characteristic that occurs in the total population with probability $p$. Notice that, before we know $E$, we do not know that the person that committed the crime has the characteristic. We write $G$ for the guilt of $A$. We are interested in $G$ given $E$.

The typical way to deal with this classically, is to use Bayes' rule that states
\begin{equation}
\label{bayes}
P(G|E) = \frac{P(E|G)P(G)}{P(E|G)P(G) + P(E|G^c)P(G^c)}.
\end{equation}
Whatever the interpretation of the classical probability measure $P$, subjective, frequentistic or otherwise, $P(G|E)$ represents the posterior probability of guilt conditioned on the available evidence, while $P(G)$ denotes the prior probability of guilt, before taking the evidence $E$ into account. From our information, we can conclude that $P(E|G)=p$ and $P(E|G^c)=p^2$. There is, however, no reason to assign any positive prior probability to either $G$ or $G^c$. But since classical probability requires that $P(G) + P(G^c)=1$, by lack of a better alternative, a uniform prior $P(G)=P(G^c)=\frac{1}{2}$ is usually taken, which leads to
\begin{equation}
\label{eq:islandklas}
P(G|E) = \frac{ \frac{1}{2}p} {\frac{1}{2}p^2 + \frac{1}{2}p} = \frac{1}{1+p}.
\end{equation}

We can re-derive the answer in (\ref{eq:islandklas}) in our setting, using the following basic belief assigment based on a uniform prior on the guilt of $A$:
\begin{equation}
\begin{aligned}
m(E \cap G) = & m \left( \; \begin{tabular}{ | l | c | c | c |}
    \hline
 & $G^c$ & $G$ \\ \hline
$E^c$ &    &     \\ \hline
$E$ &     & $\ast$         \\ \hline
\end{tabular}  \; \right) = \frac{1}{2}p, \\
m(E^c \cap G) = & m \left( \; \begin{tabular}{ | l | c | c | c |}
    \hline
 & $G^c$ & $G$ \\ \hline
$E^c$ &    & $\ast$    \\ \hline
$E$ &    &      \\ \hline
\end{tabular} \; \right) = \frac{1}{2}(1-p),  \\
m(E \cap G^c) = & m \left( \;  \begin{tabular}{ | l | c | c | c |}
    \hline
 & $G^c$ & $G$ \\ \hline
$E^c$ &     &     \\ \hline
$E$ &  $\ast$  &          \\ \hline
\end{tabular} \; \right) = \frac{1}{2}p^2,  \\
m(E^c \cap G^c) = & m \left( \;  \begin{tabular}{ | l | c | c | c |}
    \hline
 & $G^c$ & $G$ \\ \hline
$E^c$ &  $\ast$   &     \\ \hline
$E$ &    &          \\ \hline
\end{tabular} \; \right) = \frac{1}{2}(1-p^2).  \\
\end{aligned}
\end{equation}
Then we find that
\begin{equation}
\B_E(G) = \frac{ \frac{1}{2}p} {\frac{1}{2}p^2 + \frac{1}{2}p} = \frac{1}{1+p}.
\end{equation}
In our theory the problem of choosing a prior can be resolved, since belief functions are more flexible than probability distributions. We will explain this now.

What we do, is determine what we actually \emph{know} in certain scenarios. If $A$ does not have the characteristic, which happens with probability $1-p$, then we know $E^c$ but we do not know anything about the guilt of $A$. If $A$ and $B$ both have the characteristic, which happens with probability $p^2$, then we know $E$ but again we do not know anything about the guilt of $A$. If $A$ has the characteristic and $B$ has not, which happens with probability $p(1-p)$, we know that either $E^c$ or  $G$. This leads to the following basic belief assignment: 
\begin{equation}
\begin{aligned}
m(E^c) = & m \left( \; \begin{tabular}{ | l | c | c | c |}
    \hline
 & $G^c$ & $G$ \\ \hline
$E^c$ &  $\ast$   & $\ast$    \\ \hline
$E$ &    &      \\ \hline
\end{tabular} \; \right) = 1-p,  \\
m(E) = & m \left( \; \begin{tabular}{ | l | c | c | c |}
    \hline
 & $G^c$ & $G$ \\ \hline
$E^c$ &    &     \\ \hline
$E$ &   $\ast$  & $\ast$         \\ \hline
\end{tabular}  \; \right) =  p^2, \\
m(E^c \cup G) = & m \left( \;  \begin{tabular}{ | l | c | c | c |}
    \hline
 & $G^c$ & $G$ \\ \hline
$E^c$ &  $\ast$   & $\ast$    \\ \hline
$E$ &    & $\ast$         \\ \hline
\end{tabular} \; \right) = p(1-p).  \\
\end{aligned}
\end{equation}
This basic belief assignments does not give any prior belief on the guilt of $A$ since $\B(G)=\B(G^c)=0$. If we condition on $E$, however, we can simply compute, using our rule of conditioning, that
\begin{equation}
\B_{E}(G) = \frac{p(1-p)}{p(1-p)+p^2} = 1-p,
\end{equation}
which is a smaller number than the classical answer.

\subsection{A gambling example}
\label{examplegambling}

In a casino, two croupiers execute the following procedure independently of each other. First, a fair coin is flipped. After that, the croupiers get, with probability $p$, the opportunity to change the outcome of the coin flip, again independently of each other. How in such a situation the croupiers make their decisions about changing or not, is unknown by the player. After the two results are produced, but before the player knows the outcomes, the player is told whether or not the produced outcomes of the two croupiers are the same. This means that of the four possible combinations of outcomes, only two remain, and these are the two outcomes on which the player can bet.

We write $\Omega_1=\{h,t\}$ for the outcome space of the first croupier, where $h$ stands for `head' and $t$ for `tail', and we define the basic belief assignment $m_1:2^{\Omega_1} \rightarrow [0,1]$ by  $m_1(\{ h \}) = m_1(\{ t \}) = \frac{1}{2}(1-p)$ and $m_1(\{h,t\})=p$. We write $\Omega_2=\{h,t\}$ for the outcome space of the second croupier and set $m_2: 2^{\Omega_2} \rightarrow [0,1]$ similar to $m_1$. On $\Omega=\{h,t\}^2$, using our definition of independence, we get $m: 2^\Omega \rightarrow [0,1]$ given by
\begin{equation}
\begin{aligned}
& m \left( \; \begin{tabular}{ | l | c | c | c |}
    \hline
 & $h$ & $t$ \\ \hline
$h$ &  $\ast$   &     \\ \hline
$t$ &  &       \\ \hline
\end{tabular} \; \right) & = 
m \left( \; \begin{tabular}{ | l | c | c | c |}
    \hline
 & $h$ & $t$ \\ \hline
$h$ &     &     \\ \hline
$t$ & $\ast$ &       \\ \hline
\end{tabular} \; \right) \\
= & \; m \left( \; \begin{tabular}{ | l | c | c | c |}
    \hline
 & $h$ & $t$ \\ \hline
$h$ &     & $\ast$    \\ \hline
$t$ &  &       \\ \hline
\end{tabular} \; \right) & = 
m \left( \; \begin{tabular}{ | l | c | c | c |}
    \hline
 & $h$ & $t$ \\ \hline
$h$ &    &     \\ \hline
$t$ &  &  $\ast$      \\ \hline
\end{tabular} \; \right) \\
= & \; \frac{1}{4}(1-p)^2, \\
\end{aligned}
\end{equation}
\begin{equation}
\begin{aligned}
& m \left( \; \begin{tabular}{ | l | c | c | c |}
    \hline
 & $h$ & $t$ \\ \hline
$h$ &  $\ast$   &     \\ \hline
$t$ &  $\ast$ &       \\ \hline
\end{tabular} \; \right) & = 
m \left( \; \begin{tabular}{ | l | c | c | c |}
    \hline
 & $h$ & $t$ \\ \hline
$h$ &     &     \\ \hline
$t$ & $\ast$ &  $\ast$     \\ \hline
\end{tabular} \; \right) \\
= & \; m \left( \; \begin{tabular}{ | l | c | c | c |}
    \hline
 & $h$ & $t$ \\ \hline
$h$ &  $\ast$   & $\ast$    \\ \hline
$t$ &  &       \\ \hline
\end{tabular} \; \right) & = 
m \left( \; \begin{tabular}{ | l | c | c | c |}
    \hline
 & $h$ & $t$ \\ \hline
$h$ &    &  $\ast$   \\ \hline
$t$ &  &  $\ast$      \\ \hline
\end{tabular} \; \right) \\
= & \; \frac{1}{2}p(1-p), \\
\end{aligned}
\end{equation}
and
\begin{equation}
\begin{aligned}
& m \left( \; \begin{tabular}{ | l | c | c | c |}
    \hline
 & $h$ & $t$ \\ \hline
$h$ &  $\ast$   & $\ast$    \\ \hline
$t$ &  $\ast$ &  $\ast$     \\ \hline
\end{tabular} \; \right) = p^2.
\end{aligned}
\end{equation}
\medskip
We write $S=\{ (h,h), (t,t) \}$ for the event that the outcomes of the two croupiers are the same. Using our rule of conditioning, we compute our belief in $(h,h)$ in case we get the information that the results were the same:
\begin{equation}
\label{eq:theanswer}
\begin{aligned}
\B_S(\{(h,h)\}) = \frac{\frac{1}{4}(1-p)^2 + p(1-p)}{1- \frac{1}{2}(1-p)^2}.
\end{aligned}
\end{equation}
Obviously, $\B_S(\{(t,t)\})$ has the same value, and 
\begin{equation}
\B_S(\{(t,t)\}) + \B_S(\{(h,h)\}) = 1 - m_S(S) = 1 - p^2.
\end{equation}
The quantity $p^2$ can be seen as the price we have to pay for our uncertainty about the decisions of the croupiers. Notice that if $p=0$, we simply have two fair coin flips and (\ref{eq:theanswer}) equals $\frac{1}{2}$. If $p=1$, we are completely ignorant and (\ref{eq:theanswer}) equals $0$.

We now investigate how we can model this game with classical probability distributions. In that case we have to make the assumption that both croupiers make their decision according to some probability distribution.  This means that there are probabilities $p_1,p_2$ that respectively the first and second result (after possible changes of the croupiers) are head. The definition of the game implies that
\begin{equation}
\label{eq:clasint}
p_1, p_2 \in \left[ \frac{1}{2}(1-p), \frac{1}{2}(1+p) \right].
\end{equation}
Now we set $P: 2^\Omega \rightarrow [0,1]$ by $P(\{(h,h)\})=p_1p_2$, $P(\{(h,t)\})=p_1(1-p_2)$, $P(\{(t,h)\})=(1-p_1)p_2$ and $P(\{(t,t)\})=(1-p_1)(1-p_2)$. It follows that
\begin{equation}
\label{eq:theanswerclas}
P(\{(h,h)\}|S) = \frac{p_1p_2}{p_1p_2 + (1-p_1)(1-p_2)},
\end{equation}
and an easy computation shows that (\ref{eq:theanswerclas}) is contained in the interval
\begin{equation}
\label{interval}
\left[\frac{\frac12(1-p)^2}{1+p^2}, \frac{\frac12(1+p)^2}{1+p^2}\right].
\end{equation}
Note that our answer in (\ref{eq:theanswer}) is contained in this interval.

To understand the difference between (\ref{eq:theanswer}) and (\ref{eq:theanswerclas}), note that the two approaches treat conditioning fundamentally different. In a classical setting, one first needs to choose and fix $p_1$ and $p_2$ before the concept of conditioning even makes sense. In our approach with belief functions, however, we treat the uncertainty about the decisions of the croupiers on the same level as our other uncertainty, making it possible to condition without making any assumptions about the decisions of the croupiers first. 

We can make this global assessment more concrete by looking at an example. Suppose that we are in the classical situation, and that the croupiers try to get tail. That is, if first head comes up and they get the opportunity to change, then they choose tail. If tails comes up, they never change. This boils down to $p_1=p_2=\frac12(1-p)$ and corresponds to the left endpoint of the interval in (\ref{interval}).
 
Now consider the event that the first croupier flips a head and does not get the chance to revise the outcome, and the second croupier does get the chance to revise the outcome, an event with probability $\frac{1}{2}p(1-p)$. 
In the classical setting which we described, this event implies that the second croupier chooses a tail and hence the outcome $(h,t)$ is not in $S$. The mass assigned to this event, hence, only plays a role in the normalizing factor when we condition on $S$. 

In the theory of belief functions however, the conditioning works fundamentally different. Considering the same event as described above, the probability mass of this event does not end up in the normalizing factor, but is instead added to the final belief in $(h,h)$, because given $S$, it is implied that the choice of the second croupier was head.

It is an interesting question as to which answer one would choose in a real betting situation. The conditional probability of $(h,h)$ given $S$ can be safely said to be at least the left endpoint of the interval in 
(\ref{interval}), and  perhaps this is the only number someone analyzing the situation clasically, wants to use if you can bet only once. In the theory of belief functions however, one would choose for the answer in (\ref{eq:theanswer}) in case of a unique betting situation. Of course, when we repeat the betting experiment many times, one might get insight in the strategy of the croupiers, and this might be a reason to use the classical theory with appropriate values of $p_1$ and $p_2$. 

\section{The relation between belief functions and $\cal{P}_{\B}$}
\label{subsec:pbel}

Let $m: 2^\Omega \rightarrow [0,1]$ be a basic belief assignment. Let $\mathcal{P}_\B$ be the collection of probability distributions (already introduced in Section \ref{sec:basic}) on $\Omega$ that we can obtain by distributing for every $C \subseteq \Omega$ a probability mass of $m(C)$ over the elements of $C$. If $A \subseteq \Omega$, then for every $C \subseteq \Omega$ with $C \setminus A \not= \emptyset$, we can assign a probability mass of $m(C)$ to an element outside $A$. Thus we have
\begin{equation}
\label{eq:minprob2}
\inf\{  P(A) \;:\; P \in \mathcal{P}_\B \} = \sum_{C \subseteq A} m(C) = \B(A).
\end{equation}
This leads to an interpretation of belief as `minimum probability'. As we already showed in Example  \ref{ex:conditioning2}, there are $\B$ and $A,H \subseteq \Omega$ such that
\begin{equation}
\label{eq:condalt}
\B_H(A) \not = \inf \{ P(A|H) \;:\; P \in \mathcal{P}_\B \}.
\end{equation}
This makes our theory distinct from a `lower probability' theory of belief functions, see e.g.\ \cite{shafer81}. Lemma \ref{compatibel} shows how we can properly express $\B_H$ in terms of $\cal{P}_\B$.

\begin{lemma}
\label{compatibel}
$\B_H(A)=\inf \left\{ P(A|H)  \;:\; P \in \mathcal{P}_\B \; \mathrm{and} \;  P(H^c) = \B(H^c) \right\}.$
\end{lemma}

\begin{proof}
\begin{equation}
\label{eq:conditionalbeliefminimalfrequency}
\begin{aligned}
\B_H(A) & =  \frac{ \B(A \cup H^c) - \B(H^c)}{1-\B(H^c)} \\
 & =  \frac{ \inf \{ P(A \cup H^c) \;:\; P \in \mathcal{P}_\B \} - \inf \{ P(H^c) \;:\; P \in \mathcal{P}_\B \} }{1- \inf \{ P(H^c) \;:\; P \in \mathcal{P}_\B \}} \\
 & = \inf \left\{ \frac{P(A \cup H^c) - P(H^c)} {1-P(H^c)}  \;:\; P \in \mathcal{P}_\B \; \mathrm{and} \;  P(H^c) = \B(H^c) \right\} \\
  & = \inf \left\{ \frac{P(A \cap H)} {P(H)}  \;:\; P \in \mathcal{P}_\B \; \mathrm{and} \;  P(H^c) = \B(H^c) \right\} \\
  & = \inf \left\{ P(A|H)  \;:\; P \in \mathcal{P}_\B \; \mathrm{and} \;  P(H^c) = \B(H^c) \right\}.
\end{aligned}
\end{equation}
\end{proof}

Lemma \ref{compatibel} tells us that conditional belief is not the infimum over \emph{all} conditional probabilities in $\cal{P}_\B$, but only over a sub-collection of $\cal{P}_\B$. Notice that this implies that $\B_H(A) \geq \inf \{ P(A|H) \;:\; P \in \mathcal{P}_\B \}$. By only considering $P \in \mathcal{P}_\B$ with $P(H^c)=\B(H^c)$ in (\ref{eq:conditionalbeliefminimalfrequency}), we are discarding distributions of $\mathcal{P}_\B$ on the basis that we have learned $H$. This means that in our theory, the collection $\cal{P}_\B$ should be interpreted as a collection from which we can discard distributions if we have reasons to do so. In particular, this means that we can \emph{not} interpret $\cal{P}_\B$ as containing the `correct' or `actual' probability distribution, without knowing which one it is. This is because if we interpret $\cal{P}_\B$ that way, we are not allowed to discard any element of $\cal{P}_\B$, since by discarding a distribution we might discard the actual distribution.

We conclude by expressing independence in terms of $\mathcal{P}_\B$. Consider the situation of Section \ref{sec:independence} again and observe that
\begin{equation}
\begin{aligned}
& \inf\{ P(X \in A; Y \in B) \;:\; P \in \mathcal{P}_\B \}  \\
  = & \inf \{ P_1(A)P_2(B) \;:\; P_1 \in \mathcal{P}_{\B_1}, \; P_2 \in \mathcal{P}_{\B_2} \}
\end{aligned}
\end{equation}
for all  $\mathcal{A} \subseteq \Omega_1$ and $\mathcal{B} \subseteq \Omega_2$, is equivalent with the requirement of the second approach. 

\section{A law of large numbers}
\label{subsec:frequency}

Since we only have developed our theory for finite outcome spaces, we present a `weak' law of large numbers. Let $X: \Omega \rightarrow \mathbb{R}$, $m: 2^\Omega \rightarrow [0,1]$ be a basic belief assignment and $\B$ the corresponding belief function. To state our theorem, we need to generalize the concept `expectation' to our setting. There are various ways to generalize the concept, but because we aim at proving a law of large numbers, we want to define the expectation of $X$ such that it is a `guaranteed lower bound' of the average of many independent `copies' of $X$. This leads to the following definition.

\begin{definition}
\label{def:expectation}
The \emph{expectation} of $X$ (with respect to $m$) is
\begin{equation}
\E(X) := \sum_{C \subseteq \Omega} m(C) \min_{\omega \in C} X(\omega).
\end{equation}
\end{definition}

In case $\B=P$ is a probability distribution, we have 
\begin{equation}
\E(X) = \sum_{\omega \in \Omega} m(\{\omega\}) X(\omega) = \sum_{\omega \in \Omega} P(\{\omega\}) X(\omega) = E(X),
\end{equation}
and thus Definition \ref{def:expectation} is consistent with the concept of expectation for probability distributions.

First, we want to show that in the long run, $\E(A)$ is an lower bound for the average of $n$ independent `copies' of $X$. Secondly, we want to show that there is no bigger lower bound than $\E(A)$. We make this precise. On $\Omega^n$ we use Definition \ref{def:independence} to define the basic belief assignment $m_n: 2^{\Omega^n} \rightarrow [0,1]$ as
\begin{equation}
m_n(A_1 \times \cdots \times A_n) := \prod_{j=1}^n m(A_j),
\end{equation} 
making all projections independent. Let $\B_n$ be the corresponding belief function. We set $X_j: \Omega^n \rightarrow \mathbb{R}$ by
\begin{equation}
X_j( (\omega_1,\omega_2,\ldots,\omega_n)) := X(\omega_j).
\end{equation}

\begin{theorem}
\label{thm:largenumbers}
For every $\epsilon>0$ we have
\begin{equation}
\lim_{n \rightarrow \infty} \B_n \left( \frac{1}{n} \sum_{j=1}^n X_j \geq \E(X) - \epsilon  \right) = 1
\end{equation}
and
\begin{equation}
\lim_{n \rightarrow \infty} \B_n \left( \frac{1}{n} \sum_{j=1}^n X_j \geq \E(X) + \epsilon  \right) = 0.
\end{equation}
\end{theorem}

\begin{proof}
Let $\epsilon>0$ be given. We define the probability distribution $P$ by $P(\{C\}):=m(C)$ and $P_n$ by
\begin{equation}
P_n(\{ (C_1,C_2,\ldots, C_n)\}) := \prod_{j=1}^n m(C_j) = \prod_{j=1}^n P(\{C_j\}).
\end{equation}
We define the random variable $\hat{X}:2^\Omega \rightarrow \mathbb{R}$ by
\begin{equation}
\hat{X}(C) := \min_{\omega \in C} X(\omega)
\end{equation}
and let $\hat{X_j}:(2^{\Omega})^n \rightarrow \mathbb{R}$ be given by
\begin{equation}
\hat{X_j}((C_1,C_2,\ldots, C_n)) := \hat{X}(C_j).
\end{equation}
Observe that for any $\alpha \in [0,1]$ we have
\begin{equation}
\begin{aligned}
& \B_n \left( \frac{1}{n} \sum_{j=1}^n X_j \geq \alpha \right) \\
& = P_n \left(\left\{ (C_1,C_2,\ldots,C_n) \;:\; \min_{ \omega_j \in C_j } \frac{1}{n} \sum_{j=1}^n X(\omega_j) \geq \alpha \right\} \right) \\
& = P_n \left(\left\{ (C_1,C_2,\ldots,C_n) \;:\; \frac{1}{n} \sum_{j=1}^n \hat{X}(C_j) \geq \alpha \right\} \right) \\
& = P_n \left( \frac{1}{n} \sum_{j=1}^n \hat{X}_j \geq \alpha \right). \\
\end{aligned}
\end{equation}
The (classical) expectation of the $\hat{X}_j$ is
\begin{equation}
E(\hat{X}_j) = E(\hat{X}) = \sum_{C \subseteq \Omega} P(\{C\}) \hat{X}(C) = \sum_{C \subseteq \Omega} m(C) \min_{\omega \in C} X(\omega) = \E(X)
\end{equation}
and by the definition of $P_n$ all the $\hat{X}_1,\ldots,\hat{X}_n$ are (classically) independent. With the classical weak law of large numbers, we then find that
\begin{equation}
\begin{aligned}
& \lim_{n \rightarrow \infty} \B_n \left( \frac{1}{n} \sum_{j=1}^n X_j \geq \E(X)-\epsilon \right) \\
= & \lim_{n \rightarrow \infty} P_n \left( \frac{1}{n} \sum_{j=1}^n \hat{X}_j \geq E(\hat{X})-\epsilon \right) = 1
\end{aligned}
\end{equation}
and
\begin{equation}
\begin{aligned}
& \lim_{n \rightarrow \infty} \B_n \left( \frac{1}{n} \sum_{j=1}^n X_j \geq \E(X)+\epsilon \right) \\
=& \lim_{n \rightarrow \infty} P_n \left( \frac{1}{n} \sum_{j=1}^n \hat{X}_j \geq E(\hat{X})+\epsilon \right) = 0.
 \end{aligned}
\end{equation}
\end{proof}

If we take for $A \subseteq \Omega$ the random variable $1_A$, then we find
\begin{equation}
\E(1_A) = \sum_{C \subseteq \Omega} m(C) \min_{\omega \in C} 1_A(\omega) = \sum_{C \subseteq A} m(C) = \B(A).
\end{equation}
This gives us a special case of Theorem \ref{thm:largenumbers}.
\begin{lemma}[Corollary of Theorem \ref{thm:largenumbers}]
\label{lem:largenumbers}
For every $\epsilon>0$ and every $A \subseteq \Omega$ we have
\begin{equation}
\lim_{n \rightarrow \infty} \B_n \left( \frac{1}{n} \sum_{j=1}^n 1_A(\omega_j) \geq \B(A) - \epsilon  \right) = 1
\end{equation}
and
\begin{equation}
\lim_{n \rightarrow \infty} \B_n \left( \frac{1}{n} \sum_{j=1}^n 1_A(\omega_j) \geq \B(A) + \epsilon  \right) = 0.
\end{equation}
\end{lemma}

Lemma \ref{lem:largenumbers} tells us that if we write $F_n(A) \in [0,1]$ for the relative frequency of $A$ occurring after $n$ independent repetitions, then the belief $\B(A)$ is the greatest lower bound for $F_n(A)$ we can give if $n$ is large. Note, however, that this is \emph{not} the same as knowing that for every large $k$ there is a $n>k$ such that $F_n(A)$ is close to $\B(A)$. 

Lemma \ref{lem:largenumbers} provides a frequency interpretation of belief function which is analogous to the interpretation of the classical law of large numbers for classical probability theory. It gives a mathematical formulation of the intuitive idea that when we independently repeat an experiment many times, the relative frequency of the number of occurrence of an event $A$ should be related to the belief in $A$. The fact that the belief in $A$ is related to the greatest lower bound for $F_n(A)$ we can give based on our knowledge, and not to a limit of $F_n(A)$, reflects the difference between probabilities en belief functions. This difference, we say once more, is the difference between on the one hand quantifying what $F_n(A)$ {\em is}, and on the other hand quantifying what we {\em know} about $F_n(A)$.

The extent to which the law of large numbers is useful depends very much on the situation at hand. Note that we need independent repetitions of the same experiment, and in many applications in, say, legal or forensic settings, such independent repetitions do not make much sense. Nevertheless, even in these cases, it might be reassuring that under a hypothetical assumption of independent repetitions, there is a law of large numbers which reflects the nature of belief functions quite well. Section \ref{examplegambling} contains an example which is repeatable.  

\section{A betting interpretation}
\label{subsec:betting}

In a certain interpretation of probability distributions \cite{ramsey31,fin37}, the probability on a set $A$ is seen as the price an agent is willing to buy and sell a bet for that pays out $1$ if $A$ turns out to be true. Given the constraint that agents cannot assign prices in such a way that they can have a guaranteed loss (a Dutch Book), the Dutch Book Theorem tells us that probability distributions are exactly the functions that obey the Kolmogorov axioms. Here we want to give a similar interpretation for belief functions and derive a theorem similar to the Dutch Book Theorem. The crux here, is that we do not look at the price an agent is willing to buy and sell for, but only the maximum price an agent is willing to buy for. We make that idea precise.

We consider the following scenario. An agent assigns to every subset of $S \subseteq \Omega$ the maximum price $P(S) \in [0,1]$ she is willing to pay for the bet that pays out $1$ if $S$ turns out to be true. First, we look at the following theorem that gives us the constraints corresponding to probability distributions.

\begin{theorem}
\label{thm:bettingprobability}
A function $P: 2^\Omega \rightarrow [0,1]$ is a probability distribution if and only if
\begin{itemize}
\item[(P1)] $P(\Omega)=1$
\item[(P2)] For all $A_1,A_2,...,A_N \subseteq \Omega$ and $B_1,B_2,...,B_M \subseteq \Omega$ such that
\begin{equation}
\label{eq:P2condition}
\forall \omega \in \Omega \;\; \sum_{i=1}^N 1_{A_i}(\omega) \; \geq \; \sum_{j=1}^M  1_{B_j}(\omega),
\end{equation}
we have
\begin{equation}
\sum_{i=1}^N P(A_i) \geq \sum_{j=1}^M P(B_j).
\end{equation}
\end{itemize}
\end{theorem}

\begin{proof}
It is sufficient to show that (P2) is equivalent with finite additivity. First suppose (P2). Let $A,B \subseteq \Omega$ be disjoint. We have
\begin{equation}
\forall \omega \in \Omega \;\; 1_{A}(\omega) + 1_{B}(\omega) = 1_{A \cup B}(\omega),
\end{equation}
so by (P2) we find both $P(A) + P(B) \geq P(A \cup B)$ and $P(A \cup B) \geq P(A) + P(B)$. So $P$ is  finitely additive.

Now suppose that $P$ is finitely additive. Let $A_1,A_2,...,A_N \subseteq \Omega$ and $B_1,B_2,...,B_M \subseteq \Omega$ be such that (\ref{eq:P2condition}) holds. Then
\begin{equation}
\begin{aligned}
\sum_{i=1}^N P(A_i) & = \sum_{\omega \in \Omega} P(\{\omega\}) \sum_{i=1}^N 1_{A_i}(\omega)  \\ 
 & \geq \sum_{\omega \in \Omega} P(\{\omega\}) \sum_{j=1}^M 1_{B_i}(\omega)  \\
  & = \sum_{j=1}^M P(B_j). \\
\end{aligned}
\end{equation}
So (P2) holds.
\end{proof}

Constraint (P1) says that an agent always pays $1$ for bets on tautologies, i.e. $P(\Omega)=1$. Constraint (P2) says that an agent must assign prices in such a way that if she buys a set of bets that is guaranteed to pay out at least as much as another set of bets, that the total price for the first set must be as least as much as the total price of the second set of bets. 

Given our interpretation of the maximum price an agent is willing to pay for a bet, however, we think (P2) is too restrictive as illustrated by the following example. Let $\Omega=\{\omega_0,\omega_1\}$ and consider an agent that is completely ignorant about how likely $\omega_0$ or $\omega_1$ is. Of course she will be ready to pay $1$ for a bet on $\Omega$, since payout is guaranteed. But she could feel conservative in her ignorance and not be ready to pay $\emph{anything}$ for a bet on $\{\omega_0\}$ or $\{\omega_1\}$. However, if $P(\{\omega_0\})=P(\{\omega_1\})=0$ while $P(\Omega)=1$, then (P2) is violated. 

The problem is that (P2) only compares actual payout under realizations $\omega \in \Omega$. Our example shows that an agent may also be interested in guaranteed payout of a bet on $A$ if she only knows that the actual outcome is an a given set $S$. This is in line with the epistemic interpretation which we discussed earlier, since in this epistemic interpretation, a subset $S \subseteq \Omega$ corresponds to knowledge about the outcome being in $S$ without further specification.

Hence we suggest to change (\ref{eq:P2condition}) into
\begin{equation}
\forall S \subseteq \Omega \;\; \sum_{i=1}^N 1(S \subseteq A_i) \; \geq \; \sum_{j=1}^M  1(S \subseteq B_j).
\end{equation}
We then force the total price of the first set of bets to be at least the total price of the second set if not only the payout is at least as big in all cases, but also the guaranteed payout under any $S$ is at least as big in all cases. The following theorem states that if we make this change, we get a characterization of belief functions.

\begin{theorem}
\label{thm:bettingbelief}
A function $\B: 2^\Omega \rightarrow [0,1]$ is a belief function if and only if
\begin{itemize}
\item[(B1)] $\B(\Omega)=1$
\item[(B2*)] For all $A_1,A_2,\ldots, A_N \subseteq \Omega$ and $B_1,B_2,\ldots, B_M \subseteq \Omega$ such that
\begin{equation}
\label{eq:condB2*}
\forall S \subseteq \Omega \;\; \sum_{i=1}^N 1(S \subseteq A_i) \; \geq \; \sum_{j=1}^M 1(S \subseteq B_j),
\end{equation}
we have
\begin{equation}
\sum_{i=1}^N \B(A_i) \geq \sum_{j=1}^M \B(B_j).
\end{equation}
\end{itemize}
\end{theorem}

\begin{proof}
Suppose (B1) and (B2*) hold. Let $A,B \subseteq \Omega$. For all $S \subseteq \Omega$ we have
\begin{equation}
1(S \subseteq A \cup B) + 1(S \subseteq A \cap B) \geq 1(S \subseteq A) + 1(S \subseteq B).
\end{equation}
So by (B2*) we find
\begin{equation}
\B(A \cup B) + \B(A \cap B) \geq \B(A) + \B(B).
\end{equation}
So by Theorem \ref{thm:beliefalternate} $\B$ is a belief function.

Now suppose $\B$ is a belief function. Then (B1) is immediate and we have to show (B2*). Let $m: 2^\Omega \rightarrow [0,1]$ be the corresponding basic belief assignment of $\B$. Let $A_1,A_2,\ldots, A_N \subseteq \Omega$ and $B_1,B_2,\ldots, B_M \subseteq \Omega$ be such that (\ref{eq:condB2*}) holds. Then
\begin{equation}
\begin{aligned}
\sum_{i=1}^N \B(A_i) & = \sum_{S \subseteq \Omega} m(S) \sum_{i=1}^N 1(S \subseteq A_i)  \\ 
 & \geq \sum_{S \subseteq \Omega} m(S) \sum_{j=1}^M 1(S \subseteq B_j)  \\
  & = \sum_{j=1}^M \B(B_j). \\
\end{aligned}
\end{equation}
So (B2*) holds.
\end{proof}

Theorem \ref{thm:bettingbelief} tells us that an agent following the relaxed constraints, is  precisely an agent assigning $\B(A)$ as a maximum price she is willing to pay for a bet on $A$ (that payouts out $1$ if $A$ is true), for some belief function $\B$. This gives us our betting interpretation of belief functions.

\end{document}